\documentclass{amsart}
\usepackage{hyperref}
\hypersetup{bookmarksnumbered} % Adds numbers to the section titles in the PDF table of contents
\usepackage{hyperref}
\usepackage{amssymb,amsmath,amsthm}
\usepackage{graphicx}
\usepackage[subrefformat=parens,labelformat=parens]{subfig}

\usepackage{ulem}  % remove in final version
\usepackage{color} % remove in final version
\newtheorem{theorem}{Theorem}
\newtheorem*{theorem2}{Theorem 2}
\newtheorem{proposition}[theorem]{Proposition}
\newtheorem{lemma}[theorem]{Lemma}
\newtheorem{corollary}[theorem]{Corollary}

\theoremstyle{definition}

\theoremstyle{remark}
\newtheorem{remark}[theorem]{Remark}

\newcommand{\newword}[1]{\textbf{#1}}
\newcommand{\Finfty}{\mathrm{F}_{\!\hspace{0.04em}\infty}}
\newcommand{\C}{\mathfrak{C}}
\newcommand{\Z}{\mathbb{Z}}
\newcommand{\N}{\mathbb{N}}
\newcommand{\Fix}{\mathrm{Fix}}
\newcommand{\type}{\mathrm{type}}

\begin{document}

\title[Stabilizers in Higman--Thompson groups]{Stabilizers in Higman--Thompson groups}
%\author{James Belk \and Francesco Matucci}
\author{James Belk}
\address{School of Mathematics \& Statistics, University of St Andrews, St Andrews, KY16~9SS, Scotland.}
\email{\href{mailto:jmb42@st-andrews.ac.uk}{jmb42@st-andrews.ac.uk}}

\thanks{The first author has been partially supported by EPSRC grant EP/R032866/1 during the creation of this paper.}

\author{James Hyde}
\address{Department of Mathematics, Cornell University, Ithaca, New York 14853.}
\email{\href{mailto:jth263@cornell.edu}{jth263@cornell.edu}}

\author{Francesco Matucci}
\address{Dipartimento di Matematica e Applicazioni, Universit\`{a} degli Studi di Milano--Bicocca, Milan 20125, Italy.}
\email{\href{mailto:francesco.matucci@unimib.it}{francesco.matucci@unimib.it}}
\thanks{The third author is a member of the Gruppo Nazionale per le Strutture Algebriche, Geometriche e le loro Applicazioni (GNSAGA) of the Istituto Nazionale di Alta Matematica (INdAM) and gratefully acknowledges the support of the 
Funda\c{c}\~ao para a Ci\^encia e a Tecnologia  (CEMAT-Ci\^encias FCT projects UIDB/04621/2020 and UIDP/04621/2020) and of the Universit\`a degli Studi di Milano--Bicocca
(FA project ATE-2016-0045 ``Strutture Algebriche'').
}

\begin{abstract}
We investigate stabilizers of finite sets of rational points in Cantor space for the Higman--Thompson groups~$V_{n,r}$.  We prove that the pointwise stabilizer is an iterated ascending HNN extension of $V_{n,q}$ for any $q\geq 1$. We also prove that the commutator subgroup of the pointwise stabilizer is simple, and we compute the abelianization.  Finally, for each $n$ we classify such pointwise stabilizers up to isomorphism.
\end{abstract}

\maketitle

In recent work, Golan and Sapir use iterated ascending HNN extensions to describe the stabilizer of any finite set of rational points in $(0,1)$, and they classify such stabilizers up to isomorphism~\cite{GoSa}.  Here we record analogous results for stabilizers in Thompson's group~$V$, and more generally for the Higman--Thompson group~$V_{n,r}$.

For $n\geq 2$ and $r\geq 1$, let $\C_{n,r}$ denote the Cantor space $\{1,\ldots,r\}\times\{1,\ldots,n\}^\omega$, and let $V_{n,r}$ be the associated Higman--Thompson group (denoted $G_{n,r}$ in \cite{Brown,Hig2}).  A point in $\C_{n,r}$ is said to be \newword{rational} if the associated sequence is eventually repeating.  If $S$ is a finite set of rational points in $\C_{n,r}$, the 
\newword{stabilizer} $\mathrm{Stab}(S)$ is the group of elements of $V_{n,r}$ that fix $S$ setwise, the
\newword{pointwise stabilizer} $\Fix(S)$ is the subgroup of elements that fix~$S$ pointwise. Let $\Fix_0(S)$ be the subgroup consisting of elements that are the identity in some neighborhood of~$S$.  

Our main theorem is the following:

\begin{theorem} \label{thm:MainTheorem}
Let $n\geq 2$ and $q,r\geq 1$, and let $S\subset \C_{n,r}$ be a finite, nonempty set of rational points.  Then:
\begin{enumerate}
    \item $\Fix(S)$ is an iterated ascending HNN extension of~$V_{n,q}$.
    \smallskip
   \item $\Fix(S)$ and\/ $\mathrm{Stab}(S)$ have type\/~$\Finfty$.\smallskip
    \item The commutator subgroup of\/ $\Fix(S)$ is simple, and is the intersection of $\Fix_0(S)$ with the commutator subgroup of~$V_{n,r}$.\smallskip
    \item The abelianization of\/ $\Fix(S)$ is the direct sum of\/ $\Z^{|S|}$ with the abelianization of~$V_{n,r}$.
\end{enumerate}
\end{theorem}

For (1), recall that a group $G$ is an \newword{ascending HNN~extension} of a subgroup~$H$ if $G$ has presentation
\[
G = \langle H,t \mid tht^{-1}=\varphi(h)\text{ for all }h\in H\rangle
\]
where $\varphi\colon H\to H$ is a monomorphism. More generally, $G$ is an \newword{iterated ascending HNN extension} of $H$ if there exists a chain
\[
H=H_0 \leq H_1 \leq \cdots \leq H_n = G
\]
of subgroups of $H$ such that each $H_i$ is an ascending HNN extension of $H_{i-1}$.  We prove (1) in Section~\ref{sec:HNNExtensions} below.

For (2), it is well known that an ascending HNN extension of a group of type~$\Finfty$ has type $\Finfty$.  Since the groups $V_{n,r}$ all have type $\Finfty$ by a result of Brown \mbox{\cite[Theorem~7.3.1]{Brown}}, it follows from statement~(1) above that $\Fix(S)$ always has type~$\Finfty$ when $S$ is a finite set of rational points.   Since $\Fix(S)$ has finite index in~$\mathrm{Stab}(S)$, the stabilizer has type~$\Finfty$ as well, and thus (2) follows from~(1). The authors will use these results in~\cite{BHM} to prove finiteness properties for certain interesting groups, including a class of R\"{o}ver--Nekrashevych groups.

For parts (3) and (4), recall that Higman proved that the commutator subgroup of $V_{n,r}$ is always simple, and that the abelianization of $V_{n,r}$ is trivial if $n$ is even and cyclic of order two if $n$ is odd \cite[Chapter~5]{Hig2}.  Thus the abelianization of $\Fix(S)$ is $\Z^{|S|}$ if $n$ is even and $\Z^{|S|}\oplus \Z_2$ if $n$ is odd.  We prove (3) and (4) in Theorems~\ref{thm:CommutatorSubgroup} and~\ref{thm:Abelianization} below.

In addition to the above results, the following theorem classifies the isomorphism types of the groups~$\Fix(S)$.

\begin{theorem} Let $n\geq 2$ and $q,r\geq 1$, and let $S\subset \C_{n,q}$ and $S'\subset \C_{n,r}$ be finite, nonempty sets of rational points.  Then\/ $\Fix(S)$ and\/ $\Fix(S')$ are isomorphic if and only if\/~$|S|=|S'|$.
\end{theorem}

The proof of this theorem can be found at the end of Section~\ref{sec:commutator}.

\section{Background}
\label{sec:backgroud}

\subsection{Thompson-like homeomorphisms}
\label{ssec:thompson-like}
We view $\C_{n,r}$ as a space of infinite sequences of digits, where the first digit of each sequence lies in $\{1,\ldots,r\}$ and the remaining digits lie in $\{1,\ldots,n\}$.  If $\alpha\in \{1,\ldots,r\}\times\{1,\ldots,n\}^*$ is any finite prefix, the associated \newword{cone} $C_\alpha$  consists of all sequences in $\C_{n,r}$ that have $\alpha$ as a prefix.    Note that the clopen sets in $\C_{n,r}$ are precisely the finite unions of cones.

Given any two cones $C_\alpha\subseteq \C_{n,q}$ and $C_\beta\subseteq \C_{n,r}$, the corresponding \newword{prefix-replacement homeomorphism} from $C_\alpha$ to $C_\beta$ is the function that maps any word $\alpha\psi\in C_\alpha$ to the corresponding word $\beta\psi\in C_\beta$. If $U\subseteq \C_{n,q}$ and $U'\subseteq \C_{n,r}$ are open, we say that a homeomorphism $f\colon U\to U'$ is \newword{Thompson-like} if each point of $U$ is contained in a cone on which $f$ restricts to a prefix replacement.  The \newword{Higman--Thompson group $\boldsymbol{V_{n,r}}$} is the group of all Thompson-like homeomorphisms of~$\C_{n,r}$.

%Given a subset $S$ of ~$\C_{n,r}$ we say that two maps $f,g \in V_{n,r}$ \textbf{agree around $S$} if there exists a neighborhood $U_S$ of $S$ such that $f|_{U_S}=g|_{U_S}$ and we write $f \sim_S g$. By construction, $\sim_S$ is an equivalence relation. The \textbf{group of germs} $(V_{n,r})_S$ is the set
%\[
%\{ f \in V_{n,r} \mid f(s)=s, \text{ for any } s \in S\} / \sim_S
%\]
%endowed with the composition operation. If $S=\{s\}$ is a singleton, we simply write $(V_{n,r})_s$. It is easy to see that for any rational point $s$ the group $(V_{n,r})_s$ is the infinite cyclic group.

\subsection{Types of clopen sets}
\label{ssec:types}
Every clopen set $E\subseteq\C_{n,r}$ can be expressed as a finite, disjoint union $C_{\alpha_1}\uplus \cdots \uplus C_{\alpha_s}$ of cones.  Such a decomposition is not unique, since we can replace any cone $C_{\alpha_i}$ by the disjoint union of the $n$ subcones $\{C_{\alpha_ij}\}_{j=1}^n$.  However, any two decompositions of the same clopen set into cones must have the same number of cones modulo~$n-1$.  We refer to this element of $\Z_{n-1}$ as the \newword{type} of~$E$, denoted~$\mathrm{type}(E)$.  Note that any cone has type~$1$, the whole Cantor space $\C_{n,r}$ has type~$r$ (modulo~$n-1$), and
\[
\mathrm{type}(E_1\cup E_2) = \mathrm{type}(E_1) + \mathrm{type}(E_2) - \mathrm{type}(E_1\cap E_2)
\]
for any two clopen sets $E_1,E_2\subseteq \C_{n,r}$. As an immediate consequence, we see that
$\mathrm{type}(E \setminus F)=\mathrm{type}(E) - \mathrm{type}(E \cap F)$, for any clopen sets $E,F\subseteq \C_{n,r}$.

\begin{proposition}\label{prop:SameType}
Let $n\geq 2$, let $q,r\geq 1$, and let $E\subseteq \C_{n,q}$ and $E'\subseteq \C_{n,r}$ be clopen sets.  Then there exists a Thompson-like homeomorphism $f\colon E\to E'$ if and only if\/ $\type(E)=\type(E')$
\end{proposition}
\begin{proof}
If a Thompson-like homeomorphism $f:E \to E'$ exists, then there is a suitable subdivision of $E$ into cones so that their images are cones too, and therefore
$\mathrm{type}(E)=\mathrm{type}(E')$. Conversely, if $\mathrm{type}(E)=\mathrm{type}(E')$, there exists
a positive integer $m$ with
\[
m\equiv \type(E)\pmod{n-1}
\]
and suitable cones $C_1,\ldots,C_m\subseteq \C_{n,q}$ and $C_1',\ldots,C_m'\subseteq \C_{n,r}$ so that
$E=\biguplus_{i=1}^m C_i$
and $E'=\biguplus_{i=1}^m C_i'$. In this case, we can define a Thompson-like homeomorphism $f\colon E\to E'$ by mapping each $C_i$ to $C_i'$ by a prefix replacement.
\end{proof}

\begin{corollary}\label{cor:FixClopenSet}
Let $E\subset\C_{n,r}$ be a proper, nonempty clopen set, and let $q=\type(E)$.  Then\/ $\Fix(\C_{n,r}\setminus E)\cong V_{n,q}$.
\end{corollary}
\begin{proof}
By Proposition~\ref{prop:SameType}, there is a Thompson-like homeomorphism $h\colon E\to \C_{n,q}$.  Then $f\mapsto h(f|_E)h^{-1}$ defines an isomorphism $\Fix(\C_{n,r}\setminus E)\to V_{n,q}$.
\end{proof}

\subsection{Germs at fixed points}

If $p\in\C_{n,r}$ is a rational point, the \newword{group of germs} at $p$ is the quotient
\[
(V_{n,r})_p = \mathrm{Fix}(p) / \mathrm{Fix}_0(p).
\]
If $f\in \Fix(p)$, we let $(f)_p$ denote its image in $(V_{n,r})_p$.  Note that if $f,g\in \Fix(p)$, then $(f)_p=(g)_p$ if and only if $f$ and $g$ agree on a neighborhood of~$p$.

\begin{proposition}\label{prop:GroupOfGerms}
Let $p\in \C_{n,r}$ be a rational point, and write $p=\alpha\overline{\beta}$ for some finite sequences $\alpha$ and $\beta$, where $\beta$ is not a power of a shorter sequence.  Then there exists an $f\in V_{n,r}$ that agrees with the prefix replacement $\alpha\psi\mapsto \alpha\beta\psi$ in a neighborhood of~$p$, and $(V_{n,r})_p$ is the infinite cyclic group generated by~$(f)_p$
\end{proposition}
\begin{proof}
Note that $C_{\alpha\beta}$ has nonempty complement.  Since $C_{\alpha\beta}$ and $C_{\alpha\beta^2}$ have type~$1$, their complements have type~$r-1$, so by Proposition~\ref{prop:SameType} there exists a Thompson-like homeomorphism $h\colon \C_{n,r}\setminus C_{\alpha\beta} \to \C_{n,r}\setminus C_{\alpha\beta^2}$.  Let $f$ be the element of $V_{n,r}$ which agrees with $h$ on $\C_{n,r}\setminus C_{\alpha \beta}$ and maps $C_{\alpha\beta}$ to $C_{\alpha\beta^2}$ by a prefix replacement.  Then $f$ agrees with the prefix replacement $\alpha\psi\mapsto \alpha\beta\psi$ on~$C_{\alpha\beta}$.  

Now observe that if $g\in \Fix(p)$, then $g$ must act as a prefix replacement of the form $\alpha\psi\mapsto \alpha\beta^k\psi$ or $\alpha\beta^k\psi\mapsto \alpha\psi$ in a neighborhood of~$p$ for some $k\geq 0$.  In the first case, $g$ agrees with $f^k$ in a neighborhood of~$p$, and hence $(g)_p=(f)_p^k$.  In the second case, $g$ agrees with $f^{-k}$ in a neighborhood of~$p$, and hence $(g)_p=(f)_p^{-k}$. Thus $(V_{n,r})_p$ is precisely the infinite cyclic group generated by~$(f)_p$. 
\end{proof}

We refer to the germ $(f)_p$ described in Proposition~\ref{prop:GroupOfGerms} as the \newword{attracting generator} for $(V_{n,r})_p$.

We will need the following generalization of Proposition~\ref{prop:GroupOfGerms}.

\begin{proposition}\label{prop:AttractingGeneratorWithFixedSet}
Let $S\subset \C_{n,r}$ be a finite set of rational points, and let $s\in S$.  Then there exists an $f\in\Fix_0(S\setminus\{s\})$ that fixes $s$ such that $(f)_s$ is the attracting generator for~$(V_{n,r})_s$.
\end{proposition}
\begin{proof}
Let $E$ be a clopen neighborhood of $S\setminus \{s\}$ that does not contain~$s$.
Write $s=\alpha\overline{\beta}$, where $\alpha$ and $\beta$ are finite words and $\beta$ is not a power of a shorter word.  Replacing $\alpha$ with $\alpha\beta^k$ for some~$k$, we may assume that $C_\alpha$ is properly contained in $\C_{n,r}\setminus E$.  Since $\type(C_\alpha)=\type(C_{\alpha\beta})=1$, by the argument at the end of Subsection~\ref{ssec:types} we see that
\[
\type\bigl(\C_{n,r}\setminus(E\uplus C_{\alpha})\bigr) = \type\bigl(\C_{n,r}\setminus (E\uplus C_{\alpha\beta})\bigr)
\]
so by Proposition~\ref{prop:SameType} there is a Thompson-like homeomorphism
\[
h\colon \C_{n,r}\setminus(E\uplus C_{\alpha}) \to \C_{n,r}\setminus (E\uplus C_{\alpha\beta}).
\]
Then the element $f\in V_{n,r}$ that is the identity on~$E$, maps $C_\alpha$ to $C_{\alpha\beta}$ by a prefix replacement, and agrees with $h$ on $\C_{n,r}\setminus (E\uplus C_\alpha)$ has the desired properties. \end{proof}

\begin{corollary}\label{cor:AttractingAtEachPoint}
If $S\subset \C_{n,r}$ is a finite set of rational points, then there exists an $f\in \Fix(S)$ so that $(f)_s$ is the attracting generator for $(V_{n,r})_s$ for each $s\in S$.
\end{corollary}
\begin{proof}
Let $S=\{s_1,\ldots,s_m\}$.  By Proposition~\ref{prop:AttractingGeneratorWithFixedSet}, there exists for each $i$ an element $f_i\in\Fix_0(S\setminus\{s_i\})$ that fixes $s_i$ such that $(f_i)_{s_i}$ is the attracting generator for~$(V_{n,r})_{s_i}$.  Then the product $f=f_1\cdots f_m$ has the desired property.
\end{proof}

\section{Ascending HNN extensions}
\label{sec:HNNExtensions}

In this section we prove that pointwise stabilizers are ascending HNN extensions.  It follows from Bass-Serre theory (see~\cite{Serre}) that a group $G$ is an ascending HNN extension of a group $H$ if and only if $G$ acts by automorphisms on some directed tree $\Gamma$ with the following properties:
\begin{enumerate}
    \item $\Gamma$ has exactly one orbit of vertices under~$G$.\smallskip
    \item The stabilizer of each vertex of\/ $\Gamma$ is isomorphic to~$H$.\smallskip
    \item Each vertex of\/ $\Gamma$ has exactly one outgoing edge.
\end{enumerate}
The following proposition translates this geometry into algebra.

%\begin{proposition}\label{prop:AsceningHNNAction}
%Let $G$ and $H$ be groups.  Then $G$ is isomorphic to an ascending HNN extension of $H$ if and only if $G$ acts by automorphisms on some directed tree\/ $\Gamma$ with the following properties:
%\begin{enumerate}
    %\item $\Gamma$ has exactly one orbit of vertices under~$G$.\smallskip
    %\item The stabilizer of each vertex of\/ $\Gamma$ is isomorphic to~$H$.\smallskip
    %\item Each vertex of\/ $\Gamma$ has exactly one outgoing edge.\qed
%\end{enumerate}
%\end{proposition}

\begin{proposition}\label{prop:AlgebraicHNNCriterion}
Let $G$ be a group.  Let $t\in G$ have infinite order, let $H\leq G$, and suppose that
\begin{enumerate}
    \item $\langle t\rangle\cap H = 1$,\smallskip
    \item $t^{-1}Ht\leq H$, and\smallskip
    \item $\bigcup_{i,j\in\N} t^i H t^{-j}= G$.\smallskip
\end{enumerate}
Then $G$ is an ascending HNN extension of $H$ by~$t$.
\end{proposition}
\begin{proof}
Let $\Gamma$ be the directed graph whose vertices are the left cosets of $H$ in~$G$, with a directed edge from $gH$ to $gtH$ for every vertex~$gH$.  Note that if $gH=g'H$ for some $g,g'\in G$, then since $t^{-1}Ht\leq H$ we have $gtH=g'tH$, and thus $\Gamma$ has only one directed edge emanating from each vertex. Clearly $G$ acts on $\Gamma$ by automorphisms with one orbit of vertices and one orbit of edges, so it suffices to prove that $\Gamma$ is a directed tree.

Note first that $\Gamma$ is connected, for by condition (3) if $gH$ is any vertex of $\Gamma$ then $g t^j H = t^iH$ for some $i,j\in\N$, and hence $gH$ lies in the same component of $\Gamma$ as the vertex~$H$.  To prove that $\Gamma$ has no directed cycles, let $K$ be the ascending union of the subgroups $\{t^n Ht^{-n}\}_{n\in\N}$.  It follows from condition (1) that $K\cap \langle t\rangle=1$, so by condition (3) the group $G$ is the disjoint union of the cosets $\{t^j K\}_{j\in\Z}$, each of which is a disjoint union of cosets of~$H$.  Since $gtH \subseteq t^{j+1}K$ whenever $g H\subseteq t^jK$, we conclude that the graph $\Gamma$ has no directed cycles.  Since  every vertex of $\Gamma$ has only one outgoing edge, it follows that $\Gamma$ is a directed tree.\end{proof}

%Next, observe that if $p= \alpha\overline{\beta}$ is a rational point in $\C_{n,r}$, where $\alpha$ and $\beta$ are finite words and $\beta$ is not a power of any shorter word, then any element of $V_{n,r}$ that fixes $p$ must act as a prefix replacement of the form $\alpha\psi \mapsto \alpha \beta^k\psi$ or $\alpha\beta^k\psi\mapsto \alpha \psi$ in a neighborhood of~$p$.  In particular, if $f\in V_{n,r}$ and $f$ acts as $\alpha\psi \mapsto \alpha\beta \psi$ in a neighborhood of~$p$, then every element of $\Fix(\{p\})$ agrees with some power of $f$ in a neighborhood of~$p$.

We can now give a proof of the first part of Theorem~\ref{thm:MainTheorem}.

\begin{theorem}
\label{thm:ascending}
Let $n\geq 2$ and $q,r\geq 1$, and let $S$ be a finite, nonempty set of rational points in\/~$\C_{n,r}$.  Then\/ $\Fix(S)$ is an iterated ascending HNN extension of~$V_{n,q}$.
\end{theorem}

\begin{proof}
We prove the result by induction on~$|S|$.  Fix a point $s\in S$.  By Proposition~\ref{prop:AttractingGeneratorWithFixedSet} there exists an element $f\in\Fix(S)$ such that $(f)_s$ is the attracting generator for~$(V_{n,r})_s$. Let $C_\alpha$ be a cone containing $s$ on which $f$ acts as a prefix replacement $\alpha\psi\mapsto \alpha\beta\psi$, and note that $C_\alpha$ is disjoint from $S\setminus\{s\}$. Let $T$ be a clopen subset of $C_{\alpha}$ that contains $C_{\alpha\beta}$ and satisfies
\[
\type(T) \equiv r-q\pmod{n-1},
\]
and let $H = \Fix(S\cup T)$.  Since
\[
\type(\C_{n,r}\setminus T)\equiv q \pmod{n-1}
\]
it follows from Proposition~\ref{prop:SameType} that there exists a Thompson-like homeomorphism $k\colon \C_{n,r}\setminus T\to \C_{n,q}$, which determines an isomorphism $\Fix(T)\to V_{n,q}$ as in Corollary~\ref{cor:FixClopenSet}. If $|S|=1$, then $H=\Fix(T)$ and hence $H\cong V_{n,q}$. If $|S|\geq 2$ then this isomorphism instead maps $H$ to the subgroup $\Fix\bigl(k(S\setminus\{s\})\bigr)$ of $V_{n,q}$, which by our induction hypothesis is an iterated ascending HNN extension of $V_{n,q}$.  Thus it suffices to prove that $\Fix(S)$ is an ascending HNN extension of~$H$.

We verify the three conditions in Proposition~\ref{prop:AlgebraicHNNCriterion}.
Clearly $f$ has infinite order. Next, since $f^{-1}(T)$ contains $T$ we have that $f^{-1}Hf = \mathrm{Fix}(S\cup f^{-1}(T)) \leq H$.  Finally, if $g\in \Fix(S)$ then $(g)_s = (f)_s^i$ for some integer $i\in\Z$.  Let $U$ be a neighborhood of $s$ on which $g$ agrees with~$f^i$, and let $j\geq |i|$ so that $C_{\alpha \beta^j}\subseteq U$ and hence $f^j(T)\subseteq U$.  Since $f^{-i}g$ is the identity on $U$, it follows that $f^{-i-j}gf^j$ is the identity on~$T$.  Then  $f^{-i-j}gf^j$ lies in~$H$, and therefore $g\in f^{i+j}Hf^{-j}$.
\end{proof}

\section{The Commutator Subgroup}
\label{sec:commutator}

In this section we analyze the structure of the commutator subgroup $[\Fix(S),\Fix(S)]$ as well as the abelianization of~$\Fix(S)$.

%%%%%%%%%%%%%%%%%%%%%%%%%%%%%
%%%%%%%%%%%%%%%%%%%%%%%%%%%%%
\iffalse
%%%%%%%%%%%%%%%%%%%%%%%%%%%%%
%%%%%%%%%%%%%%%%%%%%%%%%%%%%%

\newword{pointwise stabilizer} of~$S$, denoted~$\Fix(S)$, is the subgroup of $V_{n,r}$ consisting of elements that fix $S$ pointwise.  

In this short note we analyze the algebraic structure of such pointwise stabilizers.

\begin{theorem}
If $S\subset \C_{n,r}$ is a finite set of rational points, then\/ $\Fix(S)$ is an iterated ascending HNN extension of~$V_{n,r}$.
\end{theorem}

\begin{corollary}
If $S\subset \C_{n,r}$ is a finite set of rational points, then\/ $\Fix(S)$ has type\/~$\Finfty$.
\end{corollary}

%%%%%%%%%%%%%%%%%%%%%%%%%%%%%
%%%%%%%%%%%%%%%%%%%%%%%%%%%%%
\fi
%%%%%%%%%%%%%%%%%%%%%%%%%%%%%
%%%%%%%%%%%%%%%%%%%%%%%%%%%%%

\begin{theorem}\label{thm:CommutatorSubgroup}Let $n\geq 2$ and $r\geq 1$, and let $S\subset \C_{n,r}$ be a finite set of rational points. Then
\[
[\Fix(S),\Fix(S)] = \Fix_0(S) \cap [V_{n,r},V_{n,r}].
\]
and this group is simple.
\end{theorem}
\begin{proof}
Let $S=\{s_1,\ldots,s_m\}$, and let $\pi\colon \Fix(S)\to \prod_{i=1}^m (V_{n,r})_{s_i}$ be the homomorphism defined by
$\pi(f) = \bigl((f)_{s_1},\ldots,(f)_{s_m}\bigr)$.
Then the kernel of $\pi$ is $\Fix_0(S)$ and the codomain is a free abelian group, so it follows that $[\Fix(S),\Fix(S)]\subseteq \Fix_0(S)$, and hence
$[\Fix(S),\Fix(S)]\subseteq \Fix_0(S)\cap [V_{n,r},V_{n,r}]$.

For the opposite inclusion, let $E_1\subseteq E_2\subseteq \cdots$ be an ascending sequence of nonempty clopen sets whose union is $\C_{n,r}\setminus S$, and let $G_i$ denote the subgroup of $V_{n,r}$ consisting of elements that are supported on~$E_i$.  Note that $\Fix_0(S)$ is the ascending union of the subgroups~$G_i$, and that each $G_i$ is isomorphic to $V_{n,t_i}$, where $t_i=\type(E_i)$.  We claim that $[G_i,G_i]=[V_{n,r},V_{n,r}]\cap G_i$ for each~$i$.

The inclusion $[G_i,G_i]\subseteq [V_{n,r},V_{n,r}]\cap G_i$ is clear.  For the opposite inclusion, if $n$ is even then $[G_i,G_i]=G_i$ and $[V_{n,r},V_{n,r}]=V_{n,r}$ and we are done.  If $n$ is odd, then $[V_{n,r},V_{n,r}]$ has index two in $V_{n,r}$, and its complement contains all elements of order two
(see \cite[Chapter~5]{Hig2}).  Then any element of order two in $G_i$ is not contained in $[V_{n,r},V_{n,r}]$, which means that $[V_{n,r},V_{n,r}]\cap G_i$ is a proper subgroup of~$G_i$.  Since $[G_i,G_i]$ has index two in $G_i$ and $[G_i,G_i]\subseteq [V_{n,r},V_{n,r}]\cap G_i$, it follows that $[G_i,G_i]=[V_{n,r},V_{n,r}]\cap G_i$.

We conclude that
\[
\Fix_0(S)\cap [V_{n,r},V_{n,r}]=\bigcup_{i\in\mathbb{N}} G_i\cap [V_{n,r},V_{n,r}] = \bigcup_{i\in\mathbb{N}} [G_i,G_i] \subseteq [\Fix(S),\Fix(S)]
\]
and therefore $[\Fix(S),\Fix(S)]=\Fix_0(S)\cap [V_{n,r},V_{n,r}]$.  Moreover, since each $G_i$ is isomorphic to $V_{n,q}$ for some $q\geq 1$, each of the commutator subgroups $[G_i,G_i]$ is simple~\cite{Hig2}, so the ascending union $[\Fix(S),\Fix(S)]$ must be simple as well.
\end{proof}

\begin{theorem}
\label{thm:Abelianization}Let $S\subseteq \C_{n,r}$ be a finite set of rational points.  Then
\[
\Fix(S)\bigr/[\Fix(S),\Fix(S)] \cong \begin{cases}\Z^{|S|} & \text{if $n$ is even,} \\ \Z^{|S|}\oplus \Z_2 & \text{if $n$ is odd.}\end{cases}
\]
\end{theorem}
\begin{proof}
Let $S=\{s_1,\ldots,s_m\}$ and let $\pi\colon \Fix(S)\to \prod_{i=1}^m(V_{n,r})_{s_i}$ be the homomorphism defined by \[
\pi(f)=\bigl((f)_{s_1},\ldots,(f)_{s_m}\bigr).
\]
Let $A = V_{n,r}/[V_{n,r},V_{n,r}]$, and let $\rho\colon V_{n,r}\to A$ be the quotient homomorphism.   Note that $\prod_{i=1}^m (V_{n,r})_{s_i}\cong \Z^m$, and that $A$ is trivial if $n$ is even and $A\cong \Z_2$ if $n$ is odd (see \cite[Chapter~5]{Hig2}).  Let $\sigma \colon \Fix(S)\to \prod_{i=1}^m (V_{n,r})_{s_i}\times A$ be the homomorphism $\pi\times \rho$.  Then \[
\ker(\sigma)=\ker(\pi)\cap \ker(\rho) = \Fix_0(S)\cap [V_{n,r},V_{n,r}] = [\Fix(S),\Fix(S)]
\]
by Theorem~\ref{thm:CommutatorSubgroup}, so it suffices to prove that $\sigma$ is surjective.

By Proposition~\ref{prop:AttractingGeneratorWithFixedSet}, there exists for each~$i$ an element $f_i\in \Fix_0(S\setminus\{s_i\})$ so that $(f_i)_{s_i}$ is a generator for $(V_{n,r})_{s_i}$.  Then $\pi$ maps $f_1,\ldots,f_m$ to a generating set for $\prod_{i=1}^m (V_{n,r})_{s_i}$, and therefore $\pi$ is surjective. Thus, all the remains is to prove that $\Fix_0(S)=\ker(\pi)$ maps onto $A$ under~$\rho$.

If $n$ is even then $A$ is trivial and we are done.  If $n$ is odd, then any element of order two in $V_{n,r}$ maps to the nontrivial element of $A$ under~$\rho$.  In particular, if we choose disjoint clopen sets $E,E'$ of the same type that are disjoint from $S$ and let $h\colon E\to E'$ be a Thompson-like homeomorphism, then the element $f\in V_{n,r}$ that agrees with $h$ on~$E$, agrees with $h^{-1}$ on~$E'$, and is the identity elsewhere lies in $\Fix_0(S)$ and maps to the nontrivial element of $A$ under~$\rho$.
\end{proof}

Finally, we are ready to prove the following

\begin{theorem2}\label{thm:Isomorphism}
Let $n\geq 2$, let $q,r\geq 1$, and let $S\subset \C_{n,q}$ and $S'\subset \C_{n,r}$ be finite sets of rational points.  Then\/ $\Fix(S)$ is isomorphic to\/ $\Fix(S')$ if and only if\/ $|S|=|S'|$.
\end{theorem2}
\begin{proof}
If $\Fix(S)\cong \Fix(S')$ then $|S|=|S'|$ by Theorem~\ref{thm:Abelianization}.  For the converse, suppose $|S|=|S'|$, and choose a bijection $\varphi\colon S\to S'$.

By Corollary~\ref{cor:AttractingAtEachPoint}, there exists an $f\in\Fix(S)$ so that for each $s\in S$, the germ $(f)_s$ is the attracting generator for $(V_{n,q})_s$.  For each $s\in S$, choose a cone containing $s$ on which $f$ acts as a prefix replacement, and let $E$ be the complement of the union of these cones. Shrinking the cones if necessary, we may assume that $E$ is nonempty.  Note then that $E$ is clopen, $f(E)\supset E$, and $\bigcup_{n\in\N} f^n(E) = \C_{n,q}\setminus S$.

Repeating the last paragraph for $S'$, we obtain an element $f'\in \Fix(S')$ and a clopen set $E'\subset \C_{n,r}$ so that $(f')_s$ is the attracting generator for $(V_{n,r})_s$ for each $s\in S'$, and moreover $f'(E')\supset E'$ and $\bigcup_{n\in\N} (f')^n(E') = \C_{n,r}\setminus S'$.  Replacing $E'$ by a larger clopen subset of $f(E')$ if necessary, we may assume that $\type(E')=\type(E)$.

Now, observe that $\C_{n,q}\setminus S$ is the disjoint union of $E$ and the sets 
\[
E_k=f^k(E)\setminus f^{k-1}(E)
\]
for $k\geq 1$.  Similarly, $\C_{n,r}\setminus S'$ is the disjoint union of $E'$ and analogous sets $E'_k$ ($k\geq 1$).  Note that $\type(E_k)=\type(E_k')=0$ for each $k\geq 1$.  Choose Thompson-like homeomorphisms $h_0\colon E\to E'$ and $h_1\colon E_1\to E_1'$, and let $h\colon \C_{n,q}\to \C_{n,r}$ be the homeomorphism that agrees with $\varphi$ on~$S$, agrees with $h_0$ on $E$, and agrees with $(f')^{n-1}h_1 f^{1-n}$ on $E_n$ for each $n\geq 1$.  We claim that $h\,\Fix(S)\,h^{-1} = \Fix(S')$, and hence $\Fix(S)\cong \Fix(S')$.

By symmetry, it suffices to prove that $h\,\Fix(S)\,h^{-1}\subseteq \Fix(S')$. If $g\in \Fix(S)$, then clearly $hgh^{-1}$ fixes $S'$ pointwise, so it suffices to prove that $hgh^{-1}$ is an element of~$V_{n,r}$.  Since $h$ restricts to a Thompson-like homeomorphism  $\C_{n,q}\setminus S\to \C_{n,r}\setminus S'$, it suffices to prove that $hgh^{-1}$ is Thompson-like in a neighborhood of~$S'$.

To see this, observe first that $hfh^{-1}$ agrees with $f'$ on $\C_{n,r}\setminus E'$.  More generally, for every $k\in\Z$ the functions $hf^kh^{-1}$ and $(f')^k$ agree in a neighborhood of~$S'$.  For each $s\in S$, we know that $(f)_s$ is a generator for the infinite cyclic group $(V_{n,q})_s$, so $g$ must agree with some power of $f$ in a neighborhood of~$s$.  It follows that $hgh^{-1}$ agrees with some power of $f'$ in a neighborhood of~$h(s)$, and therefore $hgh^{-1}$ is Thompson-like in a neighborhood of~$h(s)$.
\end{proof}

\begin{remark}
Note that $\Fix_0(S)\cong \Fix_0(S')$ for any finite sets $S$ and $S'$ (regardless of size), since any Thompson-like homeomorphism $\C_{n,q}\setminus S\to \C_{n,r}\setminus S'$ conjugates $\Fix_0(S)$ to $\Fix_0(S')$.  (More generally, if $U\subset \C_{n,q}$ and $U'\subset \C_{n,r}$ are nonempty, \mbox{non-compact} open sets, then the subgroup of $V_{n,q}$ consisting of elements supported on a compact subset of~$U$ is isomorphic to the subgroup of $V_{n,r}$ consisting of elements supported on a compact subset of~$U'$.)
Since $[\Fix(S),\Fix(S)]$ is the commutator subgroup of $\Fix_0(S)$, it follows that  $[\Fix(S),\Fix(S)]\cong [\Fix(S'),\Fix(S')]$ for any nonempty finite sets $S\subset \C_{n,q}$ and~$S'\subset \C_{n,r}$.
\end{remark}

\bigskip
\newcommand{\arXiv}[1]{\href{https://arxiv.org/abs/#1}{arXiv}}
\newcommand{\doi}[1]{\href{https://doi.org/#1}{Crossref\,}}
\bibliographystyle{plain}

\end{document}